\documentclass[reqno,10pt]{amsart}
\usepackage{amsfonts, amsmath, amssymb, amsthm, color}

\allowdisplaybreaks[1]

\newtheorem{thm}{Theorem}[section]
\newtheorem{lemma}{Lemma}[section]
\newtheorem{prop}{Proposition}[section]

\theoremstyle{definition}
 \numberwithin{equation}{section}

\newcommand{\rr}{\mathbb{R}}
\newcommand{\al}{\alpha}
\newcommand{\de}{\delta}
 \newcommand{\eps}{\epsilon}
\newcommand{\la}{\lambda}

 \renewcommand{\(}{\left(}
\renewcommand{\)}{\right)}
\renewcommand{\[}{\left[}
\renewcommand{\]}{\right]}

\begin{document}
\title[On the supercritical  mean field equation  on  pierced domains]{On the supercritical mean field equation  on   pierced domains}

\author{Mohameden Ould Ahmedou}
\address[Mohameden Ould Ahmedou] {Mathematisches Institut, Justus-Liebig-University Giessen, Ludwigstra{\ss}e 23
35390 Giessen, Germany}
\email{Mohameden.Ahmedou@math.uni-giessen.de}

 \author{Angela Pistoia}
\address[Angela Pistoia] {Dipartimento SBAI, Universit\`{a} di Roma ``La Sapienza", via Antonio Scarpa 16, 00161 Roma, Italy}
\email{pistoia@dmmm.uniroma1.it}

\begin{abstract}
We consider the problem
$$(P_\eps)\qquad
 \Delta u +\lambda  { e^{u }\over\int\limits_{\Omega\setminus B(\xi,\eps)}e^u}=0\  \hbox{in}\ \Omega\setminus B(\xi,\eps),\quad u =0\
\hbox{on}\ \partial\(\Omega\setminus B(\xi,\eps)\), $$
where     $\Omega$ is a smooth bounded open domain in $\rr^2$ which contains the point $\xi.$
We prove that if $\lambda>8\pi,$  problem $(P_\eps)$ has a solutions $u_\eps$ such that
$$u_\eps(x)\to  {8\pi+ \lambda\over2}  G(x,\xi)  \ \hbox{uniformly on compact sets of $\Omega\setminus\{\xi\}$ }$$
as $\eps$ goes to zero. Here $G$ denotes Green's function of      Dirichlet Laplacian in $\Omega.$\\
 If $\lambda\not\in 8\pi  \mathbb N$ we will not make any symmetry assumptions on $\Omega,$ while
if $\lambda \in 8\pi  \mathbb N$   we will assume that $\Omega$ is  invariant under a rotation through an angle ${ 8\pi^2\over \la}  $  around the point $\xi.$

 \end{abstract}
\subjclass{35J60, 35B33, 35J25, 35J20, 35B40}

\date{\today}

\keywords{mean field equation, blow-up solutions, supercritical problem, pierced domain} \maketitle

\footnotetext{The authors have been supported by   Vigoni Project E65E06000080001.}

\section{Introduction}

Our paper concerns with the mean field equation
\begin{equation}\label{p0}
\left\{\begin{aligned}
& \Delta u +\lambda  { e^{u }\over\int\limits_{D }e^u}=0\quad \hbox{in}\ D ,\\
   &u =0\quad
\hbox{on}\ \partial D ,\\
\end{aligned}\right.
\end{equation}

when $\la\in\mathbb R$   and $ D$ is a smooth bounded domain of $\rr^2.$
It is well known that solutions to \eqref{p0} are nothing but the critical points of the functional $J_\la:H^1_0( D)\to\rr$ defined by
$$J_\la(u):={1\over2}\int\limits_ D |\nabla u|^2-\la\ln\int\limits_ D e^u.$$
This variational problem arises from Onsager's vortex model for turbulent Euler flows. In this interpretation, $-u/\la$ is the stream function in the infinite vortex limit, whose canonical Gibbs measure and partition function are finite if $\la<8\pi.$ In this situation Caglioti et al. \cite{clmp1} and Kiessling \cite{k} proved
the existence of a minimizer of $J_\la$. Indeed,
the classical Moser-Trudinger inequality \cite{m}
$${1\over2}\int\limits_D |\nabla u|^2\ge{1\over 8\pi}\ln \int\limits_D e^{8\pi u}\ \hbox{for any} \ u\in H^1_0(D)$$
implies the compactness and coercivity properties for $J_\la$   if $\la<8\pi.$ Therefore, problem \eqref{p0} has at least a solution for any $\la<8\pi.$ \\
When $\la\ge8\pi$ the situation turns out to be more complicate, because in general $J_\la$ is no longer compact and coercive. This is a supercritical case for the Moser-Trudinger inequality.\\
The existence of solutions of \eqref{p0} actually depends on the geometry of the domain.
For example, problem \eqref{p0} has always a solution  whenever $\la\not\in8\pi\mathbb N $  and  $D$ is not simply connected, as established by Chen and Lin \cite{cl1,cl2}
 using a degree argument.\\
  If $\la\ \in8\pi\mathbb N $ the existence of solutions to problem \eqref{p0} is a delicate issue, because it does not depend  only on the topology, but also on the geometry of the domain. Indeed, when   $\la=8\pi$ problem \eqref{p0} has no solutions if $D$ is a ball and has at least one solution if $D$ is a long and thin  rectangle as showed by Caglioti et al. \cite{clmp2}.
  Recently, Bartolucci and Lin \cite{bl} proved that if  $\la=8\pi$  and the Robin's function of $D$ has more than one maximum point then problem \eqref{p0} has at least one solution. Up to our knowledge, there are no results  concerning the supercritical case $\la\ \in8\pi\mathbb N $ and $\la>8\pi.$
  \\\\

In this paper we prove that if $\lambda>8\pi$ problem  and $D$ has a small hole and is suitably symmetric, then the
problem \eqref{p0} has at least one solution.  More precisely,
we consider  the problem
 \begin{equation}\label{p}
\left\{\begin{aligned}
& \Delta u +\lambda  { e^{u }\over\int\limits_{\Omega_\eps}e^u}=0\quad \hbox{in}\ \Omega_\eps,\\
   &u =0\quad
\hbox{on}\ \partial\Omega_\eps,\\
\end{aligned}\right.
\end{equation}
where   $\lambda>8\pi$ is a   positive parameter and $\Omega_\eps:=\Omega\setminus B(\xi,\eps),$    $\eps$ is a small positive number, $\xi\in\Omega$ and  $\Omega$ is a smooth bounded domain in $\rr^2$ such that
\begin{equation}
\label{ksym}
\left.
\begin{aligned}(i)\ &\hbox{if $\la\not \in  8\pi\mathbb N,$ we will not require any symmetry assumptions on $\Omega,$} \\
& \\
(ii)\ &\hbox{if $\la  \in  8\pi\mathbb N,$  i.e. $\la=8\pi\kappa $ for some $\kappa\in\mathbb N,$ we will assume that } \\
&\hbox{$\Omega$ is $\kappa-$symmetric with respect to the point $\xi$, i.e.}\\
 &x \in\Omega-\xi\  \hbox{if and only if}\     \Re_\kappa x\in\Omega-\xi,\  \hbox{where}\ \Re_\kappa:=\(\begin{matrix}\cos{ \pi\over \kappa}&\sin{ \pi\over \kappa}\\
-\sin{ \pi\over \kappa}&\cos{\pi\over \kappa}\\ \end{matrix}\)\\ \end{aligned}\right\}
\end{equation}
 Our main result reads as follows.

 \begin{thm}\label{main} Let $\lambda>8\pi$ and assume (\ref{ksym}).
  If $\eps$ is small enough problem \eqref{p} has a   solution $u_\eps$ such that as $\eps$ goes to zero
 $$u_\eps(x)\to  {8\pi+\lambda\over2}  G(x,\xi)  \ \hbox{uniformly on compact sets of $\Omega\setminus\{\xi\}$.}$$

 \end{thm}

Here
\begin{equation}\label{green}
G(x,y)={1\over 2\pi}\ln{1\over |x-y|}+H(x,y),\quad x,y\in\Omega
\end{equation}
is the   Green's function of   Dirichlet Laplacian in $\Omega$ and $H(x,y)$ is its regular part.
The function $H(x,x)$ is the Robin's function of the domain $\Omega.$\\

 It is not clear if   the symmetry assumption (\ref{ksym}) when $\la  \in  8\pi\mathbb N $ can be removed. Indeed, let us consider the simple case $\la=8\pi$ and $\Omega_\eps:=B(0,1)\setminus B(\xi,\eps)$. If $\xi=0$ then
  problem \eqref{p0} has a radial solution (see for example Caglioti et al.   \cite{clmp2}). On the other hand if $\xi\not=0$ the problem \eqref{p0} has no solutions (see Bartolucci and Lin \cite{bl}). We point out that    $0$ is the unique critical  point of the Robin's function in the ball $B(0,1).$ This result suggests that existence of   solutions in the pierced domain $\Omega_\eps $ depends on the mutual position of the center of the hole $\xi$ and the critical points of the Robin's function of the domain $\Omega.$ Indeed, in our situation, if $\la=8\pi\kappa$ for some $\kappa\in\mathbb N$, we find a solution provided the domain $\Omega$ is symmetric with respect to the center of the hole, namely the point $\xi.$   We point out that also in this case, the center of symmetry is a critical point of the Robin's function. So, a couple of questions naturally arise.
 \\
  \textsc{Question 1:}
  \textit{does problem \eqref{p0} have a solution if $\lambda\in\{16\pi,24\pi,32\pi,\dots\}$ when   $\xi$ is not a critical point of the Robin's function of the domain $\Omega$?} \\
\textsc{Question 2:}
\textit{does problem \eqref{p0} have a solution if $\lambda\in\{16\pi,24\pi,32\pi,\dots\}$ when $\xi$ is  a critical point of the Robin's function of the domain $\Omega$, but $\Omega$ is not symmetric with respect to it?}\\

   The argument we use to find the solution relies on a simple contraction mapping argument. We set    $\al:={\la\over4\pi} $ and we look for a solution to problem \eqref{p} whose   shape resembles the bubble
\begin{equation}\label{walfa}
w^\al_\de(x):=\ln 2\al^2{\de^\al\over\(\de^\al+|x|^\al\)^2}\quad
x\in\rr^2,\ \de>0
\end{equation}
which solve the singular
Liouville problem (see Prajapat and Tarantello \cite{pt})
\begin{equation}\label{plim}
-\Delta w=|x|^{\al-2}e^w\quad \hbox{in}\quad \rr^2,\qquad
\int\limits_{\rr^2} |x|^{\al-2}e^{w(x)}dx<+\infty.
\end{equation}
If $\al$ is not an even integer, namely $\la\not\in8\pi\mathbb N,$ the linear operator ${\mathcal L}_\lambda$ introduced in \eqref{lla} is substantially invertible, while if $\al$ is an even integer, namely $\la \in8\pi\mathbb N,$ we have to  look for a solution to problem \eqref{p} in the space of
  symmetric functions according to \eqref{even}, where the linear operator ${\mathcal L}_\lambda$  is substantially invertible.
  Therefore, a direct contraction mapping argument is enough to catch the solution.

\section{The ansatz}\label{uno}

For the sake of simplicity, we will assume $\xi=0.$

Let us introduce the projection  $P_\eps  u$ of  a function $u$
into $H^1_0(\Omega_\eps),$ i.e.
\begin{equation}\label{pro}
 \Delta P _\eps u=\Delta u\quad \hbox{in}\ \Omega_\eps,\qquad  P _\eps u=0\quad \hbox{on}\ \partial\Omega_\eps.
\end{equation}

\begin{lemma}
\label{pro-exp}
Assume $\de\sim d\eps^\beta$ for some $d>0$ and $\beta\in(0,1).$
It holds true that
$$P _\eps w^\al_\de(x)=w^\al_\de(x)-\ln2\al^2\de^\al+4\pi\al H(x,0)-\gamma^\al_{\de,\eps} G(x,0)+O\(\de^\al\)+O\(\eps\)$$
where
\begin{equation}\label{gamma}
\gamma^\al_{\de,\eps}:={\ln{1\over \(\de^\al+\eps^\al\)^2}+4\pi\al H(0,0)\over{1\over2\pi}\ln{1\over \eps }+ H(0,0)}.\end{equation}
\end{lemma}
\begin{proof}
The function
$$\rho(x):=P _\eps w^\al_\de(x)-\[w^\al_\de(x)-\ln2\al^2\de^\al+4\pi\al H(x,0)-\gamma^\al_{\de,\eps} G(x,0)\] $$
solves
$-\Delta \rho=0$ in $\Omega_\eps.$ Moreover, it is easy to check that
$$\rho(x)=- \ln{1\over \(\de^\al+|x|^\al\)^2}+\ln{1\over  |x|^{2\al}}=O\(\de^\al\)\ \hbox{if}\ x\in\partial\Omega$$
and
$$\rho(x)= O\(\eps\) \ \hbox{if}\ x\in\partial B(0,\eps).$$
since the assumption on $\de$ ensures that  $\gamma^\al_{\de,\eps}=O(1).$
Therefore, the claim follows by the maximum principle.

\end{proof}

  We look for a  solution
to \eqref{p} as
\begin{equation}\label{ans}
u_\eps:=P _\eps w^\al_\de(x)+{\phi}_\eps(x), \end{equation}
  where \begin{equation}\label{alfa}
 \al ={\lambda \over 4\pi}
\end{equation}
and the concentration parameter  are chosen so that (see \eqref{gamma})
\begin{equation}\label{delta1}\gamma^\al_{\de,\eps}=2\pi (\al-2),\end{equation}  namely
\begin{equation}\label{delta2}
2\ln  \(\de^\al+\eps^\al\) -(\al-2)\ln\eps=2\pi(\al+2)H(0,0).
\end{equation}
Let us point out that by \eqref{delta2} we deduce the rate of the concentration parameter with respect to the size of the hole
\begin{equation}\label{delta0}
\de\sim e^{{\al+2\over\al}\pi H(0,0)}\eps^{ {\al-2\over2\al} }.
\end{equation}
We point out that the choice of the $\alpha$ and $\delta$'s made in  \eqref{alfa} and \eqref{delta1}
is motivated by the need that the error term defined in \eqref{rla} goes to zero as $\eps$ goes to zero.
In particular, the choice of $\delta$ made in \eqref{delta1} together with Lemma \ref{pro-exp} ensure
that
\begin{align}\label{es1}
P _\eps w^\al_\de(x)=&w^\al_\de(x)-\ln2\al^2\de^\al+4\pi \al  H(x,0)-2\pi(\al-2)G(x,0)  +O\(\eps^{\sigma_1}\)\nonumber\\
=&w^\al_\de(x)-\ln2\al^2\de^\al+2\pi(\al+2) H(x,0)+(\al-2)\ln|x|  +O\(\eps^{\sigma_1}\),
\end{align}
where $ \sigma_1:=\min\left\{{\al-2\over2},1\right\}.$

 In particular, it holds true that
\begin{equation}\label{es11}
P _\eps w^\al_\de(x)= 2\pi(\al+2) G(x,0) +o\( 1\)\ \hbox{uniformly on compact sets of $\Omega\setminus\{0\}$.}
 \end{equation}

 The rest term $ \phi_\eps $ belongs to the space $H $ defined as follows.
\begin{equation}
\label{hek}
H:=\left\{
\begin{aligned}&\mathrm{H}^1_0(\Omega_\eps)\  &\hbox{if}\  {\lambda }\not \in 4\pi\mathbb N\\
&\left\{\phi\in  \mathrm{H}^1_0(\Omega_\eps)\ :\ \phi(x)=\phi(\Re_\al x),\   x\in\Omega_\eps \right\}\ &\hbox{if}\  {\lambda } \in 4\pi\mathbb N
\end{aligned}
\right.
 \end{equation}
where $\Re_\al$ is defined in \eqref{ksym}).

\medskip

In the following, we will denote by
$$\|u\|_p:=\(\int\limits_{\Omega_\eps} |u(x)|^pdx\)^{1\over p}\quad \hbox{and}\quad\|u\| :=\(\int\limits_{\Omega_\eps}|\nabla u(x)|^2dx\)^{1\over 2}$$
the usual norms in the Banach spaces $\mathrm{L}^p(\Omega_\eps)$ and $\mathrm{H}^1_0(\Omega_\eps),$ respectively.

\section{Estimate of the error term}\label{due}

In this section we will estimate the following error term

\begin{align}\label{rla}
\mathcal{R}_\eps (x):= \Delta P _\eps w^\al_\de(x) +\lambda  { e^{P _\eps w^\al_\de(x) }\over\int\limits_{\Omega_\eps}e^{P _\eps w^\al_\de(x)}dx}.
\end{align}

 \begin{lemma}\label{errore} Let $  \mathcal{R}_\eps$ as in \eqref{rla}.
There exists   $p_0>1$ and $\eps_0>0$ such that for any $p\in(1,p_0 )$ and $\eps\in(0,\eps_0)$ we have
\begin{equation}\label{rla1}\| \mathcal{R}_\eps\|_{p}= O\({\eps^{\sigma_p}}\)\ \hbox{where}\ \sigma_p:={(\al-2)(2-p)\over2\al  p}.\end{equation}
\end{lemma}

\begin{proof}

 Now, using \eqref{es1} we can compute
 \begin{align}\label{es2}
& \int\limits_{\Omega_\eps}e^{P _\eps w^\al_\de(x)}dx=\int\limits_{\Omega_\eps}{|x|^{\al-2}\over \(\de^\al+|x|^\al\)^2}e^{2\pi(\al+2) H(x,0) +O\( \eps^{\sigma_1}\)}dx\nonumber\\
&
=\int\limits_{\Omega_\eps}{|x|^{\al-2}\over \(\de^\al+|x|^\al\)^2}e^{2\pi(\al+2) H(0,0) +O\(|x|\)+O\( \eps^{\sigma_1}\)}dx\nonumber\\
&
=\int\limits_{\Omega_\eps}{|x|^{\al-2}\over \(\de^\al+|x|^\al\)^2}e^{2\pi(\al+2) H(0,0)}dx +O\(  \int\limits_{\Omega_\eps}{|x|^{\al-2}\over \(\de^\al+|x|^\al\)^2}\(|x|+\eps^{\sigma_1}\) dx\)  \nonumber\\
&\quad\hbox{(we scale $x=\de y$)}\nonumber\\
&
={1\over\de^\al}\int\limits_{\Omega_\eps\over\de}{|y|^{\al-2}\over \(1+|y|^\al\)^2}e^{2\pi(\al+2) H(0,0)}dy +O\(  {1\over\de^\al} \int\limits_{\Omega_\eps\over\de}{|y|^{\al-2}\over \(1+|y|^\al\)^2}\(\de|y|+\eps^{\sigma_1}\) dy\)  \nonumber\\
&
={1\over\de^\al}\({2\pi\over\al}e^{2\pi(\al+2) H(0,0)}+O\(\eps^{\sigma_2}\)\)
 \end{align}
where $\sigma_2:=\min\left\{{(\al+1)(\al-2)\over2},{\al+2\over2}\right\}.$
 Indeed
 \begin{align}
 \int\limits_{\Omega_\eps\over\de}{|y|^{\al-2}\over \(1+|y|^\al\)^2}dy &=\int\limits_{\rr^2}{|y|^{\al-2}\over \(1+|y|^\al\)^2} dy
  -\int\limits_{\rr^2\setminus{\Omega \over\de}}{|y|^{\al-2}\over \(1+|y|^\al\)^2}  +  \int\limits_{B(0,\eps/\de)}{|y|^{\al-2}\over \(1+|y|^\al\)^2} dy \nonumber\\ &={2\pi\over\al}+O\(\de^{\al+1}\)+O\(\({\eps\over\de}\)^2\).
  \end{align}
By \eqref{es2} we deduce that
 \begin{align}\label{es3}
& {1\over\int\limits_{\Omega_\eps}e^{P _\eps w^\al_\de(x)}dx}= { \de^\al}\({\al\over2\pi}e^{-2\pi(\al+2) H(0,0)}+O\(\eps^{\sigma_2}\)\)  .\end{align}

Therefore, we can compute
 \begin{align*}
&R_\eps(x)=-|x|^{\al-2}e^{w^\al_\de(x)}+{\la\over\int\limits_{\Omega_\eps}e^{P _\eps w^\al_\de(x)}dx}e^{P _\eps w^\al_\de(x)}\ \hbox{(we use \eqref{es1})}\\ &=
|x|^{\al-2}e^{w^\al_\de(x)}\[-1+{\la\over 2\al^2\de^\al\int\limits_{\Omega_\eps}e^{P _\eps w^\al_\de(x)}dx} e^{2\pi(\al+2)H(x,0)+O\(\eps^{\sigma_1}\)}\]\ \hbox{(we use \eqref{es3})}\\ &=|x|^{\al-2}e^{w^\al_\de(x)}\[-1+{\la\over 4\pi\al }\(1+O\(\eps^{\sigma_2}\)\) e^{2\pi(\al+2)\[H(x,0)- H(0,0)\]+O\(\eps^{\sigma_1}\)}\]
\\ &\qquad \hbox{(we use the mean value theorem and the choice of $\al$ in \eqref{alfa})}\\ &=|x|^{\al-2}e^{w^\al_\de(x)}\left\{-1+{\la\over 4\pi\al }\[1+O\(\eps^{\sigma_2}\)\] \[ 1+O\(|x|\)+O\(\eps^{\sigma_1}\) \]\right\}\\ &\qquad \hbox{(we use the choice of $\alpha$ made in \eqref{alfa})}
\\ & =|x|^{\al-2}e^{w^\al_\de(x)}\[O\(|x|\)+O\(\eps^{\sigma_1}\) \]\end{align*}
because  $   \sigma_1=\min\left\{\sigma_1,\sigma_2\right\}.$
Finally, we get
 \begin{align*}
&\int\limits_{\Omega_\eps}\left|R_\eps(x)\right|^pdx=O\(\int\limits_{\Omega_\eps}\({|x|^{\al-1} \over \(\de^\al+|x|^\al\)^2} \)^p dx\)+O\( \int\limits_{\Omega_\eps}\(\eps^{\sigma_1}{|x|^{\al-2} \over \(\de^\al+|x|^\al\)^2} \)^p dx\)\\ &\qquad \hbox{(we scale $x=\de y$)}\\ &=O\(\de^{2-p}\int\limits_{\rr^2}\(|y|^{\al-1} \over \(1+|y|^\al\)^2 \)^p dx\)+O\( \eps^{p\sigma_1}\int\limits_{\rr^2}\(|y|^{\al-2} \over \(1+|y|^\al\)^2 \)^p dx\)\\ &\qquad (\hbox{we use \eqref{delta0} and we take $p$ close enough to 1})\\ &
=O\(\eps^{(\al-2)(2-p)\over2\al}\)+O\( \eps^{p\sigma_1}\)=O\(\eps^{(\al-2)(2-p)\over2\al}\)  ,
\end{align*}
  because ${(\al-2)(2-p)\over2\al}=\min\left\{{(\al-2)(2-p)\over2\al}, p\sigma_1\right\}$ if $p$ is close enough to 1.
 That proves our claim.

\end{proof}

\section{The linear theory}\label{tre}

It is useful to introduce   the Banach spaces
\begin{equation}\label{ljs}
\mathrm{L}_\alpha(\rr^2):=\left\{u \in {\rm W}^{1,2}_{loc}(\rr^2)\ :\  \left\|{|y|^{\alpha-2\over 2}\over 1+|y|^\alpha}u\right\|_{\mathrm{L}^2(\rr^2)}<+\infty\right\}\end{equation}
 and
\begin{equation}\label{hjs}\mathrm{H}_\alpha(\rr^2):=\left\{u\in {\rm W}^{1,2}_{loc}(\rr^2) \ :\ \|\nabla u\|_{\mathrm{L}^2(\rr^2)}+\left\|{|y|^{\alpha-2\over 2}\over 1+|y|^\alpha}u\right\|_{\mathrm{L}^2(\rr^2)}<+\infty\right\},\end{equation}
 endowed with the norms
$$\|u\|_{\mathrm{L}_\alpha}:= \left\|{|y|^{\alpha-2\over 2}\over 1+|y|^\alpha}u\right\|_{\mathrm{L}^2(\rr^2)}\ \hbox{and}\ \|u\|_{\mathrm{H}_\alpha}:= \(\|\nabla u\|^2_{\mathrm{L}^2(\rr^2)}+\left\|{|y|^{\alpha-2\over 2}\over 1+|y|^\alpha}u\right\|^2_{\mathrm{L}^2(\rr^2)}\)^{1/2}.$$
It is important to point out the compactness of the embedding $i_\al:\mathrm{H}_\alpha(\rr^2)\hookrightarrow\mathrm{L}_\alpha(\rr^2)$
  (see, for example, \cite{gp}).

Let us consider the linear operator
\begin{align}\label{lla}&\mathcal{L}_{\eps}(\boldsymbol\phi ):=  -\Delta \phi-\la{e^{P _\eps w^\al_\de} \over\int\limits_{\Omega_\eps}e^{P _\eps w^\al_\de(x)}dx}\phi
+\la{e^{P _\eps w^\al_\de}\over (\int\limits_{\Omega_\eps}e^{P _\eps w^\al_\de(x)}dx )^2} \int\limits_{\Omega_\eps}e^{P _\eps w^\al_\de(x)}\phi(x) dx \end{align}

Let us   study  the invertibility of  the linearized operator $\mathcal{L}_{\eps}.$

\begin{prop}\label{inv}
For any $p>1$ there exists $\eps_0>0$ and $c>0$ such that for any $\eps \in(0,\eps_0)$ and for any $ \psi\in  \mathrm{L}^{p}(\Omega_{\eps}) $ there exists a unique
$ \phi\in  \mathrm{W}^{2, 2}(\Omega_{\eps}) \cap  H $ solution of
$$ \mathcal{L}_{\eps}( \phi )= \psi\ \hbox{in}\ \Omega_{\eps},\  \phi=0\ \hbox{on}\ \partial\Omega_{\eps},
$$
which satisfies $$\| \phi\| \leq c |\ln\eps|  \|\psi\|_p .$$
\end{prop}
\begin{proof}
We argue by contradiction. Assume there exist $p>1,$ sequences $\eps_n\to0,$ $\psi_n\in \mathrm{L}^{p}(\Omega_n)$ and $\phi_n\in \mathrm{W}^{2, 2}(\Omega_n)$
such that
\begin{equation}\label{inv1}
-\Delta \phi_n-\la{e^{P _n w_n} \over\int\limits_{\Omega_\eps}e^{P _n w_n(x)}dx}\phi_n
+\la{e^{P _n w_n}\over (\int\limits_{\Omega_n}e^{P _n w_n(x)}dx )^2} \int\limits_{\Omega_n}e^{P _n w_n(x)}\phi_n(x) dx =\psi_n \ \hbox{in}\ \Omega_n,\ \phi_n=0\ \hbox{on}\ \partial\Omega_n,
\end{equation}
where $\Omega_n:=\Omega_{\eps_n},$ $P_n:=P_{\eps_n},$ $w_n:=w^\al_{\de_n},$ the parameters $\de_n$ being in \eqref{delta2} and
\begin{equation}\label{inv2}
\|\phi_n\| =1\quad\hbox{and}\quad    |\ln\eps_n|  \|\psi_n\|_p\to0.\end{equation}

We set $\pi_n(x):=|x|^{\al-2}e^{  w_n(x)}$ and   rewrite \eqref{inv1} by using \eqref{rla}:
\begin{equation}\label{inv1.1}
-\Delta \phi_n-\pi_n\phi_n+{1\over\la} \pi_n\int\limits_{\Omega_n}\pi_n\phi_n(x)dx =\psi_n+\rho_n \ \hbox{in}\ \Omega_n,\ \phi_n=0\ \hbox{on}\ \partial\Omega_n,
\end{equation}
where
$$\rho_n(x):=R_n(x)\phi_n(x)+{1\over\la}\[\(R_n(x)+\pi_n(x)\)\int\limits_{\Omega_n}R_n(x)\phi_n(x)dx+R_n\int\limits_{\Omega_n}\pi_n(x)\phi_n(x)dx\].$$
By Lemma \ref{errore} and \eqref{inv2} we deduce that
\begin{equation}\label{erre}
\|\rho_n\|_p=O\(\eps^\sigma\)\quad\hbox{for some }\ \sigma>0.
\end{equation}
We define
 $  \hat\phi _n(y):=\phi_n\({\de }_n y\) $ with $y\in \hat\Omega _n:={\Omega\over \de_n }.$
It solves
\begin{equation}\label{inv1.2}
-\Delta \hat\phi_n-\pi \hat\phi_n+{1\over\la} \pi \int\limits_{\hat \Omega_n}\pi(y)\hat\phi_n(y)dy =\de_n^2\(\psi_n\({\de }_n y\) +\rho_n\({\de }_n y\) \) \ \hbox{in}\ \hat\Omega_n,\ \hat\phi_n=0\ \hbox{on}\ \partial\hat\Omega_n,
\end{equation}
 where $\pi(y):=2\al ^2 {|y|^{\al -2}\over(1+|y|^{\al })^2}.$

{\em Step 1: we will show that $\hat\phi_n \to \hat\phi$ weakly in $\mathrm{H}_{\al }(\rr^2)$ and strongly in $\mathrm{L}_{\al }(\rr^2) $  with
\begin{equation}\label{step1}  \hat\phi(y)-{1\over\la}\int\limits_{\rr^2}\pi(y)\hat\phi(y)dy =a {1-|y|^{\al }\over1+|y|^{\al }}\ \hbox{  for some $ a\in\rr.$}  \end{equation}}

First of all we claim that each $\hat\phi_n$ is bounded in the space $\mathrm{H}_{\al }(\rr^2)$
defined in \eqref{hjs}.
We remark that by scaling
$$\int\limits_{\hat \Omega_n}|\nabla\hat \phi_n(y)|^2dy=\de^2_n\int\limits_{\hat \Omega_n}|\(\nabla\phi \)(\de_n y)|^2dy=\int\limits_{\Omega_n }|\nabla\phi (x)|^2dx=1.$$
Assume by contradiction that
$$  \|\hat\phi_n\|_{\mathrm{L}_{\al }(\rr^2)}^2=\int\limits_{\hat \Omega_n}\pi(y) \hat \phi_n(y) ^2dy= \int\limits_{\Omega_n }\pi_n(x) \phi^2_n (x)dx\to+\infty\quad\hbox{as}\ n\to+\infty.$$
Then, if we introduce the normalized sequence $ \hat\phi^*_n :={\hat\phi_n\over\|\hat\phi_n\|_{\mathrm{L}_{\al }(\rr^2)}}$ we have (up to a subsequence) that
$$\hat\phi^*_n\to\hat\phi^*\quad\hbox{weakly in   $\mathrm{L}_{\al }(\rr^2)$}$$
If we multiply  \eqref{inv1.2} by $\hat\phi_n$ we deduce
 \begin{align} \label{new1} 0\le\int\limits_{\hat\Omega_n }\pi(y)\hat\phi_n^2 (y)dy-{1\over\la}\(\int\limits_{\hat\Omega_n }\pi(y)\hat\phi_n  (y)dy\)^2=1+o(1). \end{align}
The R.H.S follows by   \eqref{inv2} and \eqref{erre}, while the L.H.S. follows by H\"older's inequality and the choice of $\al$ in \eqref{alfa}, namely
$${1\over\la}\(\int\limits_{\hat\Omega_n }\pi(y)\hat\phi_n  (y)dy\)^2\le {1\over\la}\(\int\limits_{\hat\Omega_n }\pi(y)  dy\)\(\int\limits_{\hat\Omega_n }\pi(y)\hat\phi_n^2  (y)dy\) =
{4\pi\al\over\la}\(\int\limits_{\hat\Omega_n }\pi(y)\hat\phi_n^2  (y)dy\)
.$$
Next, we divide \eqref{new1} by $\|\hat\phi_n\|_{\mathrm{L}_{\al }(\rr^2)}^2$, we pass to the limit and we get
 \begin{align} \label{new2}{1\over\la}\(\int\limits_{\rr^2}\pi(y)\hat\phi^*  (y)dy\)^2=1,  \end{align}
because the constant function $1\in \mathrm{L}_{\al }(\rr^2)$ and $\hat\phi^*_n\to\hat\phi^*$ weakly in $\mathrm{L}_{\al }(\rr^2)$.
On the other hand, we divide \eqref{inv1.2} by $\|\hat\phi_n\|_{\mathrm{L}_{\al }(\rr^2)}$, we pass to the limit and we get
 \begin{align} \label{new3}-\pi(y)\hat\phi^*  (y)+{1\over\la}\pi(y)\int\limits_{\rr^2}\pi(y)\hat\phi^*  (y)dy=0. \end{align}
(We use the fact that   the constant function $1\in \mathrm{L}_{\al }(\rr^2)$ and $\hat\phi^*_n\to\hat\phi^*$ weakly in $\mathrm{L}_{\al }(\rr^2)$).
Finally, by \eqref{new3} we immediately deduce that $\hat\phi^*$ is a constant function and by the choice of $\al$ in \eqref{alfa}
we get that either $ \hat\phi^*\equiv0$ or $ \hat\phi^*\equiv1,$ which contradicts \eqref{new2}.

Therefore,  each $\hat\phi_n$ is bounded in the space $\mathrm{H}_{\al }(\rr^2)$
defined in \eqref{hjs}  and (up to a subsequence)
$$\hat\phi_n(y)\to \hat\phi\quad \hbox{
weakly in $\mathrm{H}_{\al }(\rr^2)$ and strongly in $\mathrm{L}_{\al }(\rr^2)$}.$$
So we pass to the limit into \eqref{inv1.2}  and we   deduce that $\hat\phi\in H $ (see \eqref{hek}) is a solution to the equation
$$-\Delta \hat\phi=\pi \hat \phi -{1\over\la} \pi\int\limits_{\rr^2}\pi(y)\hat\phi(y)dy\ \hbox{in}\ \rr^2 .$$
Then the function $\phi_0(y):=\hat\phi(y)-{1\over\la} \int\limits_{\rr^2}\pi(y)\hat\phi(y)dy$ is a solution in the space $H $  defined in  \eqref{hek} to the linear problem
 $-\Delta \phi_0=\pi \phi_0$ in $\rr^2.$ By Theorem \ref{esposito}, we get our claim.

\medskip
{\em Step 2: we will show that   $ a=0$  in \eqref{step1}  and then either $\hat\phi\equiv0$ or $\hat\phi\equiv1$ in $\rr^2.$}

 First of all, we introduce the function
$$Z(y):={1-|y|^{\al }\over 1+|y|^{\al } }\ \hbox{and}\
  Z_n(x):=Z\({x\over\de_n}\)={\de_n^{\al }-|x|^{\al }\over \de_n^{\al }+|x|^{\al }.}$$ We know that
$Z_n$ solves (see Theorem \ref{esposito})
$$-\Delta Z_n=\pi_n Z_n\quad\hbox{in}\ \rr^2.$$
Let $P_nZ_n$ be its projection onto $\mathrm{H}^1_0(\Omega_n) $ (see \eqref{pro}), i.e.
\begin{equation}\label{cru1}
-\Delta  PZ_n=\pi_n Z_n\ \hbox{in}\ \Omega_n,\ PZ_n=0\ \hbox{on}\ \partial\Omega_n.\end{equation}
By maximum principle (see also Lemma \ref{pro-exp}) we deduce that
\begin{equation}\label{pz}P _n Z_n(x)=Z_n(x)+1- { G(x,0)\over2\pi\ln{1\over\eps_n} +H(0,0)}+O\( \eps^\sigma\),\end{equation}
for some $\sigma>0.$ Set $\gamma_n := { 2\pi\ln{1\over\eps_n} +H(0,0)}.$

We are going to show that
\begin{equation}\label{sigma}
\lim\limits_n\[\int\limits_{\Omega_n}G(x,0)\pi_n(x)\phi_n(x)dx-{1\over\la}\int\limits_{\Omega_n}G(x,0)\pi_n(x) dx\int\limits_{\Omega_n} \pi_n(x)\phi_n(x)dx\]=0.
\end{equation}
  We multiply \eqref{inv1.1} by $\gamma_n  P_nZ_n$ and \eqref{cru1} by $\gamma_n\phi.$ If we subtract the two equations obtained, we get
\begin{align*}
 & \gamma_n\int\limits_{\Omega_n}\pi_n (x)\phi_n (x)\(P_nZ_n(x)-Z_n(x)\)dx-{\gamma_n\over\la}\int\limits_{\Omega_n}\pi_n (x) P_nZ_n(x)dx\int\limits_{\Omega_n}\pi_n (x)\phi_n (x) dx\\ &=
  \gamma_n\int\limits_{\Omega_n}\(\psi_n(x)+\rho_n(x)\)P_nZ_n(x)dx,    \end{align*}
which implies together with   \eqref{pz}
\begin{align*}
 & \gamma_n\int\limits_{\Omega_n}\pi_n (x)\phi_n (x)\[1- { G(x,0)\over\gamma_n}+O\( \eps_n^\sigma\)\]dx\nonumber\\ &-{\gamma_n\over\la}\int\limits_{\Omega_n}\pi_n (x) \[Z_n(x)+1- { G(x,0)\over\gamma_n}+O\( \eps_n^\sigma\)\]dx\int\limits_{\Omega_n}\pi_n (x)\phi_n (x) dx\nonumber\\ &=
  \gamma_n\int\limits_{\Omega_n}\(\psi_n(x)+\rho_n(x)\)P_nZ_n(x)dx   \end{align*}
  and so
\begin{align}\label{cru3}
 & \gamma_n\int\limits_{\Omega_n}\pi_n (x)\phi_n (x)\[1- {1\over\la}\int\limits_{\Omega_n}\pi_n (x) \(Z_n(x)+1\)dx\]\nonumber\\ &
 -\int\limits_{\Omega_n}G(x,0)\pi_n(x)\phi_n(x)dx+{1\over\la}\int\limits_{\Omega_n}G(x,0)\pi_n(x) dx\int\limits_{\Omega_n} \pi_n(x)\phi_n(x)dx=o(1),   \end{align}
because of \eqref{inv2}, \eqref{erre} and the fact that $\gamma_n\sim|\ln\eps_n|.$
Estimate \eqref{sigma} follows by \eqref{cru3} once we   prove that
\begin{align}\label{cru3.1}
\gamma_n\int\limits_{\Omega_n}\pi_n (x)\phi_n (x)\[1- {1\over\la}\int\limits_{\Omega_n}\pi_n (x) \(Z_n(x)+1\)dx\]=o(1).   \end{align}
We have that
\begin{align}\label{cru3.2}
 {1\over\la}\int\limits_{\Omega_n}\pi_n (x) \(Z_n(x)+1\)dx &= {1\over\la}\int\limits_{\hat \Omega_n}\pi  (y) \(Z (y)+1\)dy\nonumber\\ &=
 {1\over\la}\int\limits_{\rr^2}\pi  (y)  \(Z (y)+1\)dy-{1\over\la}\int\limits_{\rr^2\setminus\hat \Omega_n}\pi  (y) \(Z (y)+1\)dy\nonumber\\ &=
 1+O\(\eps_n^{ \al-2}\).\end{align}
Indeed,   a straightforward computation leads to
 \begin{align}\label{ex1}
\int\limits_{\rr^2}\pi  (y)   Z (y)dy=\int\limits_{\rr^2}   2\al ^2{ |y|^{\al -2}\over \(1+|y|^{\al }\)^2}{  1-|y|^{\al } \over 1+|y|^{\al } }dy=0, \end{align}
 \begin{align}\label{ex0}\int\limits_{\rr^2}\pi  (y)    dy=\int\limits_{\rr^2}   2\al ^2{ |y|^{\al -2}\over \(1+|y|^{\al }\)^2} dy=4\pi\al=\la,\end{align}
because of the choice of $\al$ in \eqref{alfa}
and
$$\int\limits_{\rr^2\setminus\hat \Omega_n}\pi  (y) \(Z (y)+1\)dy=O\(\de_n^{2\al}\)=O\(\eps_n^{ \al-2}\)$$
because of   \eqref{delta0}.
Finally, \eqref{cru3.1} follows by \eqref{cru3.2} taking into account \eqref{inv2} and the fact that $\gamma_n\sim|\ln\eps_n|.$

 \medskip

Finally, we can show that $a=0 $ in \eqref{step1}.  By \eqref{sigma} we get
\begin{align*}
\int\limits_{\Omega_n}\[{1\over 2\pi}\ln|x|+H(x,0)\]\pi_n(x)\phi_n(x)dx-{1\over\la}\int\limits_{\Omega_n}\[{1\over 2\pi}\ln|x|+H(x,0)\]\pi_n(x) dx\int\limits_{\Omega_n} \pi_n(x)\phi_n(x)dx=o(1)
\end{align*}
and scaling  $y=\de_n x$ we deduce
\begin{align}\label{cru5}
&{1\over 2\pi}\[\int\limits_{\hat \Omega_n}\ln|y|\pi (y)\hat\phi_n(y)dy-{1\over\la}\int\limits_{\hat \Omega_n}\ln|y|\pi (y)\int\limits_{\hat \Omega_n} \pi (y)\hat\phi_n(y)dy\]\nonumber\\
&=-\({1\over 2\pi}\ln\de_n+H(0,0)\)\int\limits_{\hat \Omega_n} \pi (y)\hat\phi_n(y)dy\(1-{1\over\la}\int\limits_{\hat \Omega_n}\pi(y)dy\)+o(1)\nonumber\\ &
=o(1),
 \end{align}
because using the choice of $\al$ in \eqref{alfa} and arguing as in \eqref{cru3.2} it holds true that
$$1-{1\over\la}\int\limits_{\hat \Omega_n}\pi(y)dy=O\(\de_n^\al\)=O\(\eps_n^{\al-2\over\al}\).$$
On the other hand, by \eqref{step1} we can assume that the weak limit of $\hat\phi_n$ reads as $\hat\phi=aZ+b,$ for some constants $a$ and $b,$ so we pass to the limit on the L.H.S.
  of \eqref{cru5} and we get
\begin{align}\label{cru6}&\int\limits_{\hat \Omega_n}\ln|y|\pi (y)\hat\phi_n(y)dy-{1\over\la}\int\limits_{\hat \Omega_n}\ln|y|\pi (y)\int\limits_{\hat \Omega_n} \pi (y)\hat\phi_n(y)dy\nonumber\\ &=\int\limits_{\rr^2}\ln|y|\pi (y)\(aZ(y)+b\)dy-{1\over\la}\int\limits_{\rr^2}\ln|y|\pi (y)\int\limits_{\rr^2} \pi (y)\(aZ(y)+b\)dy+o(1)\nonumber\\ &
  =a\int\limits_{\rr^2}\ln|y|\pi (y)Z(y)dy+o(1)=-4a \pi +o(1),
  \end{align}
because of \eqref{ex0}, \eqref{ex1} and
\begin{align}\label{ex2}
   &\int\limits_{\rr^2}\ln|y|\pi (y)Z(y)dy=\int\limits_{\rr^2}  2\al ^2{ |y|^{\al -2}\over \(1+|y|^{\al }\)^2}{  1-|y|^{\al } \over 1+|y|^{\al } }\ln|y|dy=-4\pi.
 \end{align}
as a straightforward computation proves. Combining \eqref{cru5} and \eqref{cru6} we deduce that $a=0.$

\medskip
Finally, if $a=0$ by \eqref{step1} using the choice of $\al$ made in \eqref{alfa}, we immediately deduce that $\hat\phi$ is a constant function whose possible values are $0$  or $1.$
That concludes the proof.

{\em Step 3: we will show that a contradiction arises!}

 We multiply equation \eqref{inv1.2} by $\hat\phi_n$  and we get
\begin{align*}
1 &= \int\limits_{\hat \Omega_n} \pi(y)\hat\phi_n^2(y)dy -{1\over \la}\(\int\limits_{\hat \Omega_n} \pi(y)\hat\phi_n (y)dy \)^2+o(1)  \\
 &=\int\limits_{\rr^2} \pi(y)\hat\phi ^2(y)dy -{1\over \la}\(\int\limits_{\rr^2} \pi(y)\hat\phi  (y)dy \)^2 \ \hbox{(because  $\phi_n\to0$   strongly in
 $ \mathrm{L}_{\al }(\rr^2) $)}   \\
 &=0 \ \hbox{ if either $\hat\phi(y)\equiv0$ or $\hat\phi(y)\equiv1$ (because of the choice of $\al$ in \eqref{alfa})}
\end{align*}
and a contradiction arises!
\end{proof}

\section{A contraction mapping argument and the proof of the main theorem}\label{quattro}

First of all we point out  that $P _\eps w^\al_\de+ \phi_{\eps}$ is a solution to \eqref{p} if and only if
  $ \phi_{\eps}$ is a solution of the problem
\begin{equation}\label{L2}
 \mathcal{L}_{\la}( \phi )=\mathcal{N}_{\eps}( \phi )+ +\mathcal{R}_{\eps} \  \hbox{in}\  \Omega_\eps \\\\
\end{equation}
where the error term $\mathcal{R}_\eps$ is defined in \eqref{rla},
the linear operator $\mathcal{L}_{\la}$ is defined in \eqref{lla}   and the higher order term $\mathcal{N}_{ \eps}$ is defined as
\begin{align}\label{nla}
&\mathcal{N}_\eps( \phi ):=   \la\[{e^{P _\eps w^\al_\de+\phi} \over\int\limits_{\Omega_\eps}e^{P _\eps w^\al_\de(x)+\phi(x)}dx}
-  {e^{P _\eps w^\al_\de  } \over\int\limits_{\Omega_\eps}e^{P _\eps w^\al_\de(x) }dx}
- {e^{P _\eps w^\al_\de} \over\int\limits_{\Omega_\eps}e^{P _\eps w^\al_\de(x)}dx}\phi
+ {e^{P _\eps w^\al_\de}\over (\int\limits_{\Omega_\eps}e^{P _\eps w^\al_\de(x)}dx )^2} \int\limits_{\Omega_\eps}e^{P _\eps w^\al_\de(x)}\phi dx \]
\end{align}

\begin{prop}\label{resto}
There exists $p_0>0,$  $\eps_0>0$ and $R_0>0$ such that for any    $p\in(1,p_0),$ $ \eps\in(0,\eps_0)$ and $R\ge R_0$    there exists a unique solution $ \phi_\eps  \in  H $   to the equation
\begin{equation}\label{eqrid}
  \Delta (P _\eps w^\al_\de+\phi_\eps)+ \la {e^{P _\eps w^\al_\de+\phi_\eps} \over\int\limits_{\Omega_\eps}e^{P _\eps w^\al_\de(x)+\phi_\eps(x)}dx}=0\ \hbox{in}\ \Omega_\eps, \phi =0\ \hbox{on}\ \partial\Omega_\eps
 \end{equation}
   and (see \eqref{rla1})
$$\|\phi_\eps\|\le R  \eps^{\sigma_p}|\ln\eps|. $$

\end{prop}

\begin{proof}

 As a consequence of Proposition \ref{inv}, we conclude that $ \phi$
is a solution to \eqref{eqrid} if and only if it is a fixed point
for the operator $ \mathcal{T}_\eps:H\to H,$ defined by
$$\mathcal{T}_\eps( \phi)= \left( \mathcal{L}_\eps\right)^{-1}
\left(\mathcal{N}_\eps( \phi)+  +\mathcal{R}_\eps\right),$$
where $\mathcal{L}_\eps$,  $\mathcal{N}_\eps$ and $\mathcal{R}_\eps$  are defined  in \eqref{rla},    \eqref{nla}  and \eqref{rla}, respectively.

 Let   us introduce the ball $\mathcal B_{\eps,R}:=\left\{ \phi\in H\ :\ \| \phi\|\le R\eps^{\sigma_p}|\ln\eps| \right\}$. We will show that $\mathcal{T}_\eps:\mathcal B_{\eps,R}\to \mathcal B_{\eps,R}$ is a contraction mapping
provided $\eps$ is small enough and $R$ is large enough.

\medskip {\em Let us prove that $\mathcal{T}_\eps$ maps  the ball $\mathcal B_{\eps,R}$ into itself,  i.e.}
\begin{equation}\label{c2.1}
\| \phi\|\le R\eps^{\sigma_p} |\ln\eps|\ \Longrightarrow\
\left\|\mathcal{T}_\eps( \phi)\right\|\le R\eps^{\sigma_p} |\ln\eps|.
\end{equation}

By     Lemma \ref{B2} (where we take $h=\mathcal{N}_\eps(\phi)+ \mathcal{R}_\eps$), by \eqref{B21} and by Lemma  \ref{errore} we deduce that:
\begin{align*}
\left\|\mathcal{T}_\eps( \phi)\right\|&  \le c|\ln\eps|  \left(\left\|\mathcal{N}_\eps ( \phi)\right\|_p + \left\| \mathcal{R}_\eps\right\|_p\right)
 \le c|\ln\eps|\( \eps^{\sigma'_p} \| \phi\|^2+  \eps^{\sigma _p} \)
 \le  R\eps^{\sigma _p} |\ln\eps|
\end{align*}
provided   $p$ is close enough to 1,  $R$  is suitable large and $\eps $  is  small enough. That proves \eqref{c2.1}.

\medskip {\em Let us prove that $\mathcal{T}_\eps$ is a contraction mapping,  i.e. there exists $\ell>1$ such that}
\begin{equation}\label{c2.2}
\|\phi\|\le R\eps^{\sigma_p} |\ln\eps\ \Longrightarrow\
\left\|\mathcal{T}_\eps(\ \phi_1)-\mathcal{T}_\eps(\phi_2)\right\|\le
\ell \|\phi_1-\phi_2\|.
\end{equation}

By     Lemma \ref{B2} (where we take $\psi=\mathcal{N}_\eps( \phi_1 )-\mathcal{N}_\eps( \phi_2 ) $) and by  \eqref{B22}, we deduce that:
\begin{align*}
&\left\|\mathcal{T}_\eps( \phi)\right\|   \le c|\ln\eps| \(\left\|\mathcal{N}_\eps( \phi_1 )-\mathcal{N}_\eps( \phi_2 )\right\|_p
  \)
\le c\eps^{\sigma'_p} |\ln\eps|
  \| \phi_1- \phi_2\| \( \| \phi_1\|+\|  \phi_2\|\)    \le  \ell\| \phi_1- \phi_2\| \end{align*} for some $\ell<1,$
provided   $p$ is close enough to 1,    $R$  is suitable large and $\eps $  is  small enough. That proves \eqref{c2.2}.

\end{proof}

  \begin{lemma}\label{B2}
    There exists $p_0>1$ and $\eps_0>0$   such that for any $p \in(1,p_0),$   $\eps\in(0,\eps_0)$ and $R>0$ we have
  for any $\phi,\phi_1,\phi_2\in\mathcal B_{\eps,R}:=\{\phi\in \mathrm{H}^1_0(\Omega) \ :\ \|\phi\|\le R\eps^{\sigma_p}|\ln\eps|\}$
 \begin{equation}\label{B21}
 \left\|\mathcal N_\eps( \phi)\right\|_p=O\( \eps^{\sigma'_p}\| \phi \|^2\)\end{equation}
  and
    \begin{equation}\label{B22}
\left\|\mathcal N_\eps( \phi_1)-\mathcal N_\eps( \phi_2)\right\|_p=O\(
 \eps^{\sigma'_p}
   \| \phi_1- \phi_2\|\(\| \phi_1\|+\| \phi_2\|\)\) ,\end{equation}
 for some $\sigma'_p>0.$ \end{lemma}
\begin{proof}

    Since   \eqref{B21} follows by   \eqref{B22} choosing
$\phi_2=0 ,$ we only prove (\ref{B22}). We point
out that
$$\mathcal N_\eps( \phi)=\lambda \[f(\phi)-f(0)-f'(0) (\phi)\],\ \hbox{where}\ f(\phi):={e^{P _\eps w^\al_\de+\phi} \over\int\limits_{\Omega_\eps}e^{P _\eps w^\al_\de(x)+\phi(x)}dx}.$$
Therefore, we apply the mean value theorem, we set $\phi_\theta=\theta\phi_1+(1-\theta)\phi_2 $ and $\phi_\eta=\eta\phi_1+(1-\eta)\phi_2 $ for some $\theta,\eta\in[0,1]$
\begin{align}\label{ne1}
\mathcal N_\eps( \phi_1)-\mathcal N_\eps( \phi_2)&=\lambda  \left\{f(\phi_1)-f(\phi_2)-f'(0) [\phi_1-\phi_2]\right\}
 =\lambda  \left\{f'(\phi_\theta)-f'(0)\right\}   [\phi_1-\phi_2]\nonumber \\ &=\lambda   f''(\phi_\eta) [\phi_\theta,  \phi_1-\phi_2]\end{align}
where
\begin{align*}
f''(u)[\phi,\psi]&={e^u\over\int\limits_{\Omega_\eps}e^u}\phi\psi-{e^u\over\(\int\limits_{\Omega_\eps}e^u\)^2}\phi\int\limits_{\Omega_\eps}e^u\psi-{e^u\over\(\int\limits_{\Omega_\eps}e^u\)^2}\psi\int\limits_{\Omega_\eps}e^u\phi\\
 &-{e^u\over\(\int\limits_{\Omega_\eps}e^u\)^2}\int\limits_{\Omega_\eps}e^u\psi\phi+2{e^u\over\(\int\limits_{\Omega_\eps}e^u\)^3}\int\limits_{\Omega_\eps}e^u\psi
 \int\limits_{\Omega_\eps}e^u\psi.\end{align*}

We use H\"older's inequalities
$$\|uv \|_p\le \|u\|_{pr}\|v\|_{ps},\ {1\over r }+{1\over s }=1\quad\hbox{or}\quad
\|uvw\|_p\le \|u\|_{pr}\|v\|_{ps}\|w\|_{pt},\  {1\over r }+{1\over s }+{1\over t }=1$$
and we get
\begin{align}\label{ne2}
\left\|f''(u)[\phi,\psi]\right\|_p&\le {\|e^u\|_{pr}\over\|e^u\|_{1}}\|\phi\|_{ps}\|\psi\|_{pt}+2 {\|e^u\|^2_{pr}\over\|e^u\|^2_{1}}\|\phi\|_{ps}\|\psi\|_{ps}\nonumber \\
 & +{\|e^u\|_{p }\over\|e^u\|^2_{1}}\|e^u\|_{pr}\|\phi\|_{ps}\|\psi\|_{pt}+2 {\|e^u\|_{p }\over\|e^u\|^3_{1}}\|e^u\|^2_{pr}\|\phi\|_{ps}\|\psi\|_{ps}\nonumber \\
 & \le c \( {\|e^u\|_{pr}\over\|e^u\|_{1}}+ {\|e^u\|^2_{pr}\over\|e^u\|^2_{1}}+ {\|e^u\|^3_{pr}\over\|e^u\|^3_{1}}\)\|\phi\|\|\psi\|\end{align}
because $L^{pr}(\Omega_\eps)\hookrightarrow L^{p}(\Omega_\eps)$ for any $r\ge1$ and $L^{q}(\Omega_\eps)\hookrightarrow H^1_0(\Omega_\eps)$ for any $q>1.$
Now, we put together \eqref{ne1} and \eqref{ne2} with $u=\phi_\eta=\eta\phi_1+(1-\eta)\phi_2,$ $\phi=\phi_\theta=\theta\phi_1+(1-\theta)\phi_2 $  and
$\psi= \phi_1-\phi_2.$
It only remains to estimate ${\|e^u\|_{pr}\over\|e^u\|_{1}}.$
 First of all,  arguing exactly as in the proof of \eqref{es2}
we can prove that
\begin{equation}\label{ne3}
\left\|e^{P_\eps w^\al_\de}\right\|_q=O\(\de^{{2\over q}-(2+\alpha)}\)\ \hbox{for any}\ q\ge1.\end{equation}
 Moreover, \eqref{es2} implies that
\begin{equation}\label{ne4}\left\|e^{P_\eps w^\al_\de}\right\|_1\ge {c_0\over \de^\al}\ \hbox{for some}\ c_0>0.\end{equation}
On the other hand, using the estimate $|e^a -1|\le |a| $ for any $a\in\rr$ we have
\begin{align}\label{ne5}
&\left\|e^{P_\eps w^\al_\de+\phi_\eta}-e^{P_\eps w^\al_\de }\right\|_{q}=
 \(\int\limits_{\Omega_\eps}\left| e^{P_\eps w^\al_\de+\phi_\eta}-e^{P_\eps w^\al_\de}\right|^{q}dx\)^{1/q} = \(\int\limits_{\Omega_\eps}\left|e^{P_\eps w^\al_\de}\right|^{q}\left| e^{ \phi_\eta}-1\right|^{q}dx\)^{1/q} \nonumber\\ & \le  \(\int\limits_{\Omega_\eps}\left|e^{P_\eps w^\al_\de}\right|^{q}\left| \phi_\eta\right|^{q}dx\)^{1/q}
\ \hbox{(because of the estimate $|e^a -1|\le |a| $ for any $a\in\rr$ )}
\nonumber\\ & \le\left\|e^{P_\eps w^\al_\de}\right\|_{qs}\left\| \phi_\eta\right\|_{qt}\ \hbox{(we use H\"older's estimate with ${1\over sr}+ {1\over t}=1$)}
\nonumber\\ & \le c  \de^{{2\over qs}-(2+\alpha)}\left\| \phi_\eta\right\| \ \hbox{(because of \eqref{ne3} and the fact that $L^{qt}(\Omega_\eps)\hookrightarrow H^1_0(\Omega_\eps)$)}
\nonumber\\ & \le c  \de^{ {2\over qs}-(2+\alpha) }\eps^{\sigma_p}|\ln\eps| \ \hbox{(because   $\phi_\eta\in \mathcal B_{\eps,R}$)}.\end{align}
In particular,  if $q=1$ we get
\begin{equation}\label{ne6}
\left\|e^{P_\eps w^\al_\de+\phi_\eta}-e^{P_\eps w^\al_\de }\right\|_1=O\( \de^{ {2\over  s}-(2+\alpha) }\eps^{\sigma_p}|\ln\eps|\)\ \hbox{for any $s>1.$}\end{equation}
 By \eqref{ne3}   and \eqref{ne5} we get
 \begin{equation}\label{ne7}\left\|e^{P_\eps w^\al_\de+\phi_\eta} \right\|_{q}=O\(\de^{ {2\over qs}-(2+\alpha) }\eps^{\sigma_p}|\ln\eps| \)\end{equation}
 and by \eqref{ne4} and \eqref{ne6} taking into account \eqref{delta0} and choosing $s$ close enough to 1, we get
\begin{equation}\label{ne8}\left\|e^{P_\eps w^\al_\de+\phi_\eta} \right\|_{1}\ge {c_0\over \de^\al}-c\de^{ {2\over  s}-(2+\alpha) }\eps^{\sigma_p}|\ln\eps|\ge
 {1\over \de^\al}\(c_0-c\eps^{{\al-2\over2\al}\({2\over s}-2 \)+\sigma_p}|\ln\eps|\)\ge{c_0\over 2\de^\al} .\end{equation}

Finally, by \eqref{ne7} with $q=pr$ and \eqref{ne8} taking into account \eqref{delta0} we deduce
 \begin{equation}\label{ne9}{\left\|e^{P_\eps w^\al_\de+\phi_\eta}\right\|_{pr}\over\left\|e^{P_\eps w^\al_\de+\phi_\eta}\right\|_1}=O\( \de^{ {2\over  prs}-2 }\eps^{\sigma_p}|\ln\eps|\)
=O\( \eps^{{\al-2\over2\al}\({2\over prs}-2 \)+\sigma_p}|\ln\eps|\)=O\( \eps^{\sigma'_p}|\ln\eps|\)\end{equation}
where the exponent
$\sigma'_p:={\al-2\over2\al}\({2\over prs}-2 \)+\sigma_p>0$ if $p,r,s$ are close enough to 1.

Now, we can conclude the proof. By \eqref{ne1}, \eqref{ne2} and \eqref{ne9} we get
 \begin{align}\label{ne10}
\left\|\mathcal N_\eps( \phi_1)-\mathcal N_\eps( \phi_2)\right|_{p}&\le c   \left\| f''(\phi_\eta) [\phi_\theta,  \phi_1-\phi_2]\right\|
 \le c {\left\|e^{P_\eps w^\al_\de+\phi_\eta}\right\|_{pr}\over\left\|e^{P_\eps w^\al_\de+\phi_\eta}\right\|_1} \|\phi_\theta\| \left\|\phi_1-\phi_2 \right\|\nonumber\\ &
\le   \eps^{ \sigma'_p}|\ln\eps| \(\left\|\phi_1\right\|+\left\|\phi_2 \right\|\)\left\|\phi_1-\phi_2 \right\|\end{align}
that proves our claim.

\end{proof}

\begin{proof}[Proof of Theorem \eqref{main} completed]
The existence of a solution  $u_\eps=P_\eps w^\al_\de+\phi_\eps$ follows directly by Proposition \eqref{resto}. The asymptotic shape of the solution $u_\eps$ as $\eps$ goes to zero follows by \eqref{es11}.
\end{proof}

  \appendix
  \section{ }
  In the study of the linear theory we used in a crucial way the following results.
 \begin{thm}
\label{esposito}
Let $\phi$ a   solution to   the equation
\begin{equation}\label{l1}
-\Delta \phi =2\alpha^2{|y|^{\alpha-2}\over (1+|y|^\alpha)^2}\phi\ \hbox{in}\ \rr^2,\quad \int\limits_{\rr^2}|\nabla \phi(y)|^2dy<+\infty.
\end{equation}
If $\al=2\kappa$ for some $\kappa\in\mathbb N,$ we also require that $\phi$ is $\kappa-$symmetric with respect to the origin, i.e.
 \begin{equation}\label{even}\phi(y)=\phi(\Re_\kappa y)\ \hbox{for any     $y\in\rr^2$,}\end{equation}
   where $\Re_\kappa$ is defined in \eqref{ksym}.
Then
$$\phi(y)=\gamma   {1-|y|^\al\over 1+|y|^\al}\ \hbox{for some}\ \gamma\in\rr .$$
\end{thm}
\begin{proof}

If $\al $ is an even integer, Del Pino-Esposito-Musso in \cite{dem} proved that all the bounded solutions to \eqref{l1} are a linear combination of the following functions (which are written in polar coordinates)
\begin{equation}\label{tree}\phi_0(y):   ={1-|y|^\alpha\over 1+|y|^\alpha} ,\  \phi_1(y):={ |y|^{\alpha\over 2}\over 1+|y|^\alpha} \cos{\alpha\over2}\theta
,\
\phi_2(y):={ |y|^{\alpha\over 2}\over 1+|y|^\alpha} \sin{\alpha\over2}\theta.
\end{equation}
We observe that $\phi_0$   always satisfies \eqref{even}, while the functions $\phi_1$ and $\phi_2$ do not satisfy \eqref{even}.
When $\al $ is not an even integer, the situation is easier, since only the function $\phi_0$ generates the set of solutions to the linear equation \eqref{l1}.
In \cite{gp} it was proved  that any solution $\phi$ of \eqref{l1} is actually a bounded solution. That concludes  the proof.
\end{proof}

\end{document}